\documentclass[reqno,12pt]{amsart}
\usepackage{url}
\newtheorem{theorem}{Theorem}[section]

\newtheorem{corollary}[theorem]{Corollary}
\newtheorem{definition}[theorem]{Definition}
\newtheorem{lemma}[theorem]{Lemma}
\newtheorem{proposition}[theorem]{Proposition}
\theoremstyle{remark}

\newtheorem{remark}[theorem]{Remark}
\numberwithin{equation}{section}

\newcommand{\vanish}[1]{}\parskip=12pt

\newcommand{\wt}{v}
\newcommand{\imi}{{\mathbf i}}

\begin{document}
\title[Meixner polynomials and quantum algebras]{Meixner polynomials of the second kind and quantum algebras representing
  $su(1,1)$} 
\author[G.\ Hetyei]{G\'abor Hetyei}
\email{ghetyei@uncc.edu}
\address{Department of Mathematics and Statistics, University of
  North Carolina at Charlotte, Charlotte, NC 28223, USA}


\maketitle

\begin{abstract}
We show how Viennot's combinatorial theory of orthogonal polynomials may
be used to generalize some recent results of Sukumar and Hodges on the
matrix entries in powers of certain operators in a representation of
$su(1,1)$.  
Our results link these calculations to finding the moments and inverse
polynomial coefficients of certain Laguerre polynomials and
Meixner polynomials of the second kind. As an immediate consequence 
of results by Koelink, Groenevelt and Van Der Jeugt, for the related operators,
substitutions into essentially the same Laguerre polynomials
and Meixner polynomials of the
second kind may be used to express their eigenvectors. Our
combinatorial approach explains and generalizes this ``coincidence''.
\end{abstract}

\section*{Introduction}

In two recent papers (Hodges \& Sukumar 2007; Sukumar \& Hodges 2007),
Sukumar and Hodges introduced a one-parameter family of operator algebras
exhibiting a parity-dependent structure, to study a quantum harmonic
oscillator. The operators $R$ and $L$, 
arising as the half of the square of the creation and annihilation
operators, respectively, generate a subalgebra representing
the Lie algebra $su(1,1)$. This representation is the
direct sum of two discrete series representations. 
Sukumar and Hodges observed that calculating the matrix entries in
the powers of $L+R$ lead to some interesting combinatorial
questions. In this paper we show that these results may be
generalized in such a way that the answer may be expressed in terms of
moments and inverse polynomial coefficients associated to
Laguerre polynomials and Meixner polynomials of the second
kind. 

Almost at the same time, Klimyk (2006) used discrete series representations 
of $su(1,1)$ to study another model of a quantum harmonic oscillator.  
In Klimyk's setting, the operator $L+R$
corresponds  to the momentum operator, and Klimyk (2006) noted that the
eigenfunctions of this operator may be expressed via substitutions into
{\em Meixner-Pollaczek polynomials} with the appropriate parameters. 
This observation is a consequence of the more general results of
Van Der Jeugt (1997), Koelink \& Van Der Jeugt (1998) and
Groenevelt \& Koelink (2002), describing the spectra of certain
self-adjoint elements in an arbitrary discrete series representation of
$su(1,1)$, where besides the Meixner-Pollaczek polynomials, the 
Laguerre polynomials and the Meixner polynomials of the first kind also
appear, depending on the choice of the parameters. 

In this paper we generalize and explain this remarkable ``coincidence'', 
using Viennot's (1983) combinatorial theory of orthogonal
polynomials. For that purpose we will consider tridiagonal operators $T$
defined on a subspace of $\ell^2({\mathbb Z}_{\geq 0})$ that contains at
least the subspace generated by the basis vectors $\{e_0,
e_1, e_2,\ldots\}$. 

In section~\ref{S_3d} we express 
the matrix entries $\langle e_{i+d}, T^m e_i\rangle$ in terms of
weighted Motzkin paths. Using Viennot's (1983) theory, in the case when
the matrix of $T$ has nonzero off-diagonal entries wherever this is
allowed, we are able to associate an orthogonal
polynomial sequence $\{p_n(x)\}_{n\geq 0}$ to $T$ in such a way that the
matrix entries $\langle e_{0}, T^m e_{0}\rangle$ and $\langle e_{d}, T^m
e_{0}\rangle$ may be expressed in terms of the moments and inverse polynomial
coefficients associated to $\{p_n(x)\}_{n\geq 0}$. Applications of these
general results to the discrete series representations of $su(1,1)$
follow in section~\ref{S_su11adapt}.  

In section~\ref{S_spectrum} we see that for those tridiagonal 
operators $T$ that have nonzero off-diagonal matrix entries wherever this is
allowed, the same orthogonal polynomial sequence $\{p_n(x)\}_{n\geq 0}$
as in section~\ref{S_3d} may be associated to $T$ to express its
(potential) eigenvectors. For the discrete series
representations of $su(1,1)$ we thus recover the necessity part of the
above mentioned results of Van Der Jeugt (1997), Koelink \& Van Der
Jeugt (1998) and Groenevelt \& Koelink (2002). 

The ideas presented in this paper also provide a combinatorial framework 
to discuss some representations of
other Lie (super)algebras over $\ell^2({\mathbb Z}_{\geq 0})$. 
In the concluding section~\ref{S_ex} we outline a
sample application. The other example in section~\ref{S_ex} indicates
that we do not have to limit ourself to self-adjoint operators at this 
level of investigation. It seems harder to replace the Hilbert space
$\ell^2({\mathbb Z}_{\geq 0})$ with another Hilbert space, as in that
case one would be facing the dilemma, what notion should replace the
notion of the associated orthogonal polynomial sequences. Finding such a
generalization seems an interesting challenge for future research.

\section{Preliminaries}

\subsection{Viennot's combinatorial proof of Favard's theorem}
\label{s_Vfav}

Concerning orthogonal polynomials, in this paper we follow
Viennot's (1983) notation and terminology, but sometimes we 
complement the facts stated therein with results cited from
Chihara's (1978) classical work. Viennot's work (1983) is currently out
of print but available on the author's webpage. Some
of his most important results were also communicated by Stanton (2003),
a recommended source for those readers whose French is not fluent. 

The direct way to define an {\em orthogonal polynomial sequence (OPS)}
$\{p_n(x)\}_{n\geq 0}$ is to provide a linear form $f: {\mathbb
  C}[x]\rightarrow {\mathbb C}$ and postulate the following three axioms:
\begin{itemize}
\item[(i)] for all $n$, $p_n(x)$ is a polynomial of degree $n$,
\item[(ii)] $f(p_m(x)p_n(x))=0$ if $m\neq n$,
\item[(iii)] for all $n$, $f(p_n(x)^2)\neq 0$.
\end{itemize}
The map $f$ is called a {\em moment functional}. Whenever an OPS exists,
each of its elements is determined up to a non-zero constant factor, see
the corollary of theorem 2.2 in ch.\ I of Chihara's (1978) book.   

An equivalent way define an OPS is by means of Favard's
theorem, see theorems 4.1 and 4.4 in ch.\  I of Chihara's (1978) book. 
This states 
that a sequence of monic polynomials $\{p_n(x)\}_{n\geq 0}$ is an
orthogonal polynomial sequence,
if and only if it satisfies the initial conditions
\begin{equation}
\label{E_Favinit}
p_0(x)=1, p_1(x)=x-b_0,
\end{equation}
and a two-term recurrence formula  
\begin{equation}
\label{E_Favrec}
p_{n+1}(x)=(x-b_n)p_{n}(x)-\lambda_n p_{n-1}(x) \quad\mbox{for $n\geq 1$,}
\end{equation}
where the numbers $b_n$ and $\lambda_n$ are constants and $\lambda_n\neq 0$
 for $n\geq 1$. The above form appears in theorem
 9 of ch.\ I in Viennot's (1983) book, in Chihara's (1978) work the
 indices are shifted.  
The original proof of Favard's theorem 
provides only a recursive description of the moment associated 
 functional $f$. In proposition 17 of ch.\ I, 
Viennot (1983) gave an explicit combinatorial description by expressing
 the {\em moments} $\mu_n:=f(x^n)$ as sums of weighted {\em Motzkin paths}. A
 Motzkin path of length $n$, as defined in def.\ 15 of ch.\ I by Viennot
 (1983), is a path
 $\omega=(s_0,\ldots,s_n)$ in  $\mathbb Z\times \mathbb Z$ from
 $s_0$ to $s_n$ such that the second coordinate of all $s_i$'s
 is non-negative, and each step $(s_i,s_{i+1})$ is either a northeast
 step $(1,1)$ or an east step $(1,0)$ or a southeast step $(1,-1)$.  
Viennot (1983) introduces the following valuation: the
 weight $\wt(\omega)$ of $\omega$ is the product of the weights of its
 steps $(s_{i-1},s_i)$, where each northeast step has weight $1$, each 
east step at level $k$ has weight $b_k$ and each southeast step
starting at level $k$ has weight $\lambda_k$. Here the level is the
second coordinate of the lattice point. He then defines
$\mu_n:=\sum_{\omega} \wt(\omega)$ as the total weight of all Motzkin
 paths $\omega$ from $(0,0)$ to $(n,0)$. 
\begin{theorem}[Viennot (1983)]
\label{T_Vmain}
Let $\{p_n(x)\}_{n\geq 0}$ be a sequence of monic polynomials, given by 
(\ref{E_Favinit}) and (\ref{E_Favrec}) and 
let $f: {\mathbb C}[x]\rightarrow {\mathbb C}$ be the linear map given
by $f(x^n):=\mu_n$. Then for all $n,k,\ell\geq 0$ we have
$$
f(x^n p_k(x)p_{\ell}(x))=\lambda_1\cdots \lambda_{\ell}\sum_{\omega} \wt(\omega)
$$ 
where the summation is over all Motzkin paths of length $n$ from level
$k$ to level $\ell$. 
\end{theorem}
Substituting $n=0$ in theorem~\ref{T_Vmain} yields the more difficult
implication of Favard's theorem. A generalization of 
theorem~\ref{T_Vmain} is theorem 1 in ch.\ III of Viennot's (1983) work, which 
allows the computation of the {\em inverse polynomials} of a system of
monic polynomials  $\{p_n(x)\}_{n\geq 0}$ given by (\ref{E_Favinit}) and
(\ref{E_Favrec}), even if some scalars $\lambda_n$ are zero (and thus
$\{p_n(x)\}_{n\geq   0}$ is not an OPS). The inverse polynomials
$q_n(x):=\sum_{i=0}^n q_{n,i} x^i$ are defined by $x^n=\sum_{i=0}^n
q_{n,i} p_i(x)$.  
\begin{theorem}[Viennot (1983)]
\label{T_Vinverse}
Let $\{p_n(x)\}_{n\geq 0}$ be a system of monic polynomials defined 
by (\ref{E_Favinit}) and (\ref{E_Favrec}) for some numbers
$\{b_n\}_{n\geq 0}$ and $\{\lambda_n\}_{n\geq 1}$. The coefficient
$q_{n,k}$ of $x^k$ in the inverse polynomial $q_n(x)$ is then the
total weight of all weighted Motzkin paths of length $n$ starting at
$(0,0)$ and ending at $(n,k)$. 
\end{theorem}

\subsection{Viennot's ``histoires'', Laguerre and Meixner polynomials}

Sometimes models of even deeper combinatorial interest 
may be built, using Viennot's (1983) second valuation of Motzkin paths, defined
as follows. Given the sequences of numbers $\{a_k\}_{k\geq 0}$, $\{b_k\}_{k\geq
  0}$ and $\{c_k\}_{k\geq 1}$, we define the weight
$\wt_1(\omega)$ of a Motzkin path $\omega$ as the product of the weights
$\wt_1(s_{i-1},s_i)$ of its steps, such that each northeast
(resp.\ east, resp.\ southeast) step starting at level $k$ has weight
$a_k$ (resp.\ $b_k$, resp.\ $c_k$). Setting 
\begin{equation}
\label{E_ldef}
\lambda_{k}=a_{k-1}c_k\quad\mbox{for $k\geq 1$},
\end{equation}
for a Motzkin path $\omega$ from $(0,0)$ to $(n,0)$,
$\wt_1(\omega)$ is equal to $\wt(\omega)$, as defined in
subsection~\ref{s_Vfav}, because each 
northeast step starting at some level $k-1$ may be matched to the 
first subsequent southeast step starting at level $k$. More generally,
for a Motzkin path $\omega$ starting at level $k$ and ending at level
$l$, we have 
\begin{equation}
\label{E_hist}
\wt(\omega)=
\begin{cases}
\sum_{\omega} \wt_1(\omega)&\mbox{if $k=l$,}\\
a_k a_{k+1}\cdots a_{l-1}\sum_{\omega} \wt_1(\omega)&\mbox{if $k<l$,}\\
c_{k}c_{l-1}\cdots c_{l+1}\sum_{\omega} \wt_1(\omega)&\mbox{if $k>l$,}\\
\end{cases}
\end{equation}
see lemma 1 in ch.\ II of Viennot (1983). Algebraically,
eq.\ (\ref{E_ldef}) corresponds to switch from the three term recurrence
relation for monic orthogonal polynomials to a general three term recurrence 
relation. The combinatorial
interest in $\wt_1$ arises when the sequences
$\{a_k\}_{k\geq 0}$, $\{b_k\}_{k\geq 0}$ and $\{c_k\}_{k\geq 1}$ consist
of non-negative integers, allowing us to think of these weights as
making {\em choices} 
from a set of options. In such situations Viennot (1983) defines a {\em
  story (``histoire'')} as a pair $(\omega; (p_1,\ldots,p_n))$ of a
Motzkin path $\omega=(s_0,\ldots,s_n)$ and a sequence
$(p_1,\ldots,p_n)$ of positive integers satisfying 
$1\leq p_i\leq \wt_1(s_{i-1},s_i)$ for $1\leq i\leq n$. 
Clearly the moment $\mu_n$ is the number of stories of length $n$. 
Thus, in such situations, we may replace the weighted lattice path model
with a model that involves enumerating (non-weighted) combinatorial
objects. In particular, permutation enumeration models of the (monic) Laguerre
polynomials and of the Meixner polynomials of the second kind (discussed
below) may be obtained by specializing Viennot's (1983) bijection between his
``histoires de Laguerre'' of length $n$ associated to the valuation
\begin{equation}
\label{E_Lhist}
a_k=k+1,\quad b_k=2k+2\quad\mbox{for $k\geq 0$;} \quad c_k=k+1
\quad\mbox{for $k\geq 1$,}
\end{equation}
and the permutations of the set $\{1,\ldots,n+1\}$. For the details we
refer to Viennot's (1983) work. 

The {\em (monic) Laguerre polynomials $L_n^{(\alpha)}(x)$} are
the OPS defined by (\ref{E_Favinit}) and (\ref{E_Favrec}) where
$b_n=2n+\alpha+1$ and $\lambda_n=n(n+\alpha)$, see \S
  5 in ch.\ II of Viennot's (1983) work. The associated moments are given in
  (31') of ch.\ II by Viennot's (1983):
\begin{equation}
\label{E_Lmoment}
\mu_n=(\alpha+1)_n=(\alpha+1)(\alpha+2)\cdots (\alpha+n).
\end{equation}  
The combinatorial model mentioned above is associated to the case
$\alpha=1$ yielding {\em large Laguerre stories}. For general
$\alpha$ we have {\em weighted Laguerre stories}, for $\alpha=0$ we obtain
{\em restricted Laguerre stories} by limiting $b_k'$ to $k$
above. The restricted Laguerre stories are thus a subset of the
large Laguerre stories and correspond bijectively to the
permutations $\sigma\in S_{n+1}$ with $\sigma(1)=n+1$. 

Another class of particular interest to us are the {\em Meixner
  polynomials of the first kind $m_n(x;\delta,\eta):=(c-1)^n/c^n
  \hat{m}_n(x;\beta,c)$}, whose monic variant
  $\{\hat{m}_n(x;\beta,c)\}_{n\geq 0}$ is defined by  
  (\ref{E_Favinit}) and (\ref{E_Favrec}) where 
\begin{equation}
\label{E_mx1rec}
b_n=\frac{(1+c)n+\beta c}{1-c}\quad\mbox{and}\quad
\lambda_n=\frac{cn(n+\beta-1)}{(1-c)^2}.
\end{equation} 
Viennot (1983) (see (54'') in ch.\ 2) )has a combinatorial proof of
the fact that the moments of the corresponding linear functional are given by
\begin{equation}
\label{E_mx1mom}
\mu_n(\beta,c)=(1-c)^{\beta}\sum_{k\geq 0} k^n c^k\frac{(\beta)_k}{k!}.
\end{equation}
Here $(\beta)_k=\beta(\beta+1)\cdots (\beta+k-1)$.

Finally we will be interested in the {\em (monic) Meixner
  polynomials of the second kind $M_n(x;\delta,\eta)$}, defined by
  (\ref{E_Favinit}) and (\ref{E_Favrec}) where 
\begin{equation}
\label{E_mx2rec}
b_n=(2n+\eta)\delta\quad\mbox{and}\quad
\lambda_n=(\delta^2+1)n(n+\eta-1).
\end{equation} 
A combinatorially interesting
  special case is $\delta=0$, implying $b_n=0$ and $\lambda_n=n(n+\eta-1)$. The
  moments associated to $\{M_n(x;\delta,\eta)\}_{n\geq 0}$ may be
  expressed using the Motzkin paths associated to the Laguerre
  polynomials $\{L_n^{(\eta-1)}(x)\}_{n\geq 0}$ subject to the
  restriction that {\em no east step occurs}. Motzkin paths with no east 
 steps are also known as {\em Dyck paths}. Using the
 permutation-enumeration model mentioned above, Viennot (1983) (in
 ch.\ 2, example 21) shows that  $\mu_n(0,\eta)$ is given by 
\begin{equation}
\label{E_m2n}
\mu_n(0,\eta)=\sum_{\sigma\in Z_n}\eta^{{\rm s}(\sigma)}.
\end{equation}
Here $Z_n$ is the set of {\em alternating or zig-zag} permutations of
$\{1,2,\ldots,n,\}$, starting and ending with a 
rise, and ${\rm s}(\sigma)$ is the number of {\em left-to-right minima
  (``\'el\'ements saillants'')} $\sigma(i)$, defined by 
$\sigma(i)=\min\{\sigma(1),\sigma(2),\ldots,\sigma(i)\}$. 
A generating function for the numbers $\mu_{2n}(0,\eta)$ for positive
integer $\eta$ was already computed by Carlitz (1959). He considered the
polynomials $f_n^{(k)}(x)=(-\imi)^n 
  M_n(\imi x;0,k)$ for $\imi=\sqrt{-1}$, but using the substitution
  $x\mapsto x/\imi$ we may recover the moments $\mu_{2n}(0,\eta)$. In the
  next lemma we state Carlitz' (1959) result (see his formulas (9.3) and
  (9.13)) modified for our purposes, together 
  with the outline of a combinatorial proof that extends its validity to
  all numbers $\eta$. 
\begin{lemma}
\label{L_Carlitz}
For all $\eta\neq 0$ the moments of $\{M_n(x;0,\eta)\}_{n\geq 0}$ satisfy
$$
\mu_{2n}(0,\eta)=E^{(\eta)}_{2n}\quad\mbox{where}\quad
\sec ^{\eta} (t)=\sum_{n\geq 0} E^{(\eta)}_{2n} \frac{t^{2n}}{(2n)!}.
$$
\end{lemma}
\begin{proof}
Let $E_{2n,k}$ denote the number of those zig-zag permutations on $2n$
elements that start and end with a rise and have $k$ left-to-right minima. 
Since $1$ is the rightmost left-to-right minimum,
using the decomposition $\sigma(1)\cdots\sigma(2n)=\sigma(1)\cdots
1\cdots \sigma(2n)$ we have the recursion formula
\begin{equation}
\label{E_prec}
E_{2n,k}=\sum_{m=0}^{n-1} \binom{2n-1}{2m} E_{2m,k-1}\cdot T_{2n-2m-1}.
\end{equation}
Here $T_{2n-2m-1}$ counts the number of zig-zag
permutations on $2n-2m-1$ elements, starting with a descent and ending
with a rise. Introducing 
$\phi_k(t):=\sum_{n\geq 0} E_{2n,k} t^{2n}/(2n)!$
we obtain the recursion 
$$
\phi_k(t)=\int_0^t \phi_{k-1}(u) \tan(u)\ du. 
$$
Using this, we may show by induction that
$\phi_k(t)=(-\ln\cos(t))^k/k!$. Thus, by (\ref{E_m2n}),
$\mu_{2n}(0,\eta)$ is the coefficient of $t^{2n}/(2n)!$ in
$$
\sum_{k\geq 0} (-1)^k \ln\cos(t)^k\cdot \eta^k/k!
=\exp(-\eta\ln\cos(t))=\sec(t)^{\eta}.$$ 
\end{proof}

A variant of the Meixner polynomials of the second kind are the {\em
  Meixner-Pollaczek polynomials $\{P_n^{(\lambda)}(x;\phi)\}_{n\geq
  0}$}, defined by  (1.7.1) in the Askey-Wilson scheme by
  Koekoek \& Swarttouw (1998). We should note for future reference that 
the two OPS are essentially the same. The formula connecting them is
  (3.22) in ch.\ VI of Chihara's (1978) work:  
\begin{equation}
\label{E_MPdef}
P_n^{\lambda}(x;\phi)=\frac{\sin^n \phi}{n!}
M_n(2x;\delta,2\lambda),\quad \delta=\cot \phi,\quad 0<\phi<\pi.
\end{equation}
Chihara (1978) does not use the adjective ``Meixner-Pollaczek'', 
but gives the same recurrence for them in 
(5.13) of his ch.\ VI as the recurrence (1.7.3) of Koekoek \& Swarttouw (1998). Setting $\phi=\pi/2$ in
~(\ref{E_MPdef}) leads to $\delta=0$. According to (1.7.4) of Koekoek \&
Swarttouw (1998), the normalized version of $P_n^{\lambda}(x;\phi)$ is
$n!/(2\sin \phi)^n P_n^{\lambda}(x,\phi)$. As a consequence of eq.\
(\ref{E_MPdef}) we obtain the following. 
\begin{corollary}
\label{C_MPn}
For $0<\phi<1$, the normalized version of the Meixner-Pollaczek polynomials
$\{P_n^{\lambda}(x;\phi)\}_{n\geq 0}$ is
$\{2^{-n} M_n(2x,\cot \phi,2\lambda)\}_{n\geq 0}$. 
\end{corollary}

Finally we review the inverse polynomials of
    $\{L_n^{(\alpha)}(x)\}_{n\geq 0}$, $\{\hat{m}_n(x;\beta,c)\}_{n\geq
    0}$ and $\{M_n(x;\delta,\eta)\}_{n\geq
    0}$. All three OPS are examples of {\em Sheffer orthogonal polynomials}
    defined as polynomials $\{p_n(x)\}_{n\geq 0}$, given by a generating function
\begin{equation}
\label{E_Sdef}
\sum_{n\geq 0} p_n(x) \frac{t^n}{n!}=f(t)\exp(x g(t)).
\end{equation}
As noted in (11) and (12) of ch.\ III in Viennot's (1983) work, the inverse
polynomials $\{q_n(x)\}_{n\geq 0}$ of a Sheffer OPS $\{p_n(x)\}_{n\geq
  0}$ form a Sheffer OPS, given by
\begin{equation}
\label{E_Sinv}
\sum_{n\geq 0} q_n(x) \frac{t^n}{n!}
=\frac{1}{f(g^{\langle -1\rangle }(t))}\exp(x g^{\langle -1\rangle }(t)).
\end{equation}
Here $g^{\langle -1\rangle }(t)$ stands for the compositional inverse of
$g(t)$. 

\subsection{Discrete series representations of $su(1,1)$ and two applications}
\label{SS_su11ur}

In this subsection we review the discrete series representations of
the Lie algebra $su(1,1)$, together with two applications of them
in models of a quantum oscillator. Our notation will be close to the one
used by Groenevelt \& Koelink (2002). 

The Lie algebra $su(1,1)$ has the generators $H,B,C$, subject to the
commutation relations $[H,B]=2B$, $[H,C]=-2C$ and $[B,C]=H$. According
to section 6.4 in the work of Vilenkin and Klimyk (1991), there are four
classes of irreducible unitary representations of $su(1,1)$: the positive
and negative discrete series representations, the principal series
representations and the complementary series representations. In this
paper we will focus on the two discrete series representations, defined
on the representation space  
$\ell^2({\mathbb Z}_{\geq 0})$, whose orthonormal basis vectors we will 
denote by $\{e_n\}_{n\geq 0}$. The positive discrete series
representations $\pi^+_k$ are labeled by $k>0$. The action is given by 
\begin{equation}
\label{E_psr}
\begin{array}{rcl}
\pi^+_k(H)e_n&=& 2(k+n)e_n,\\ 
\pi^+_k(B)e_n&=&\sqrt{(n+1)(2k+n)}e_{n+1} \quad\mbox{and}\quad\\
\pi^+_k(C)e_n&=&-\sqrt{n(2k+n-1)}e_{n-1}, 
\end{array}
\end{equation}
see eq.\ (2.2) of Groenevelt \& Koelink (2002). 
The negative discrete series
representations $\pi^-_k$ are labeled by $k>0$. The action is given by 
\begin{equation}
\label{E_psn}
\begin{array}{rcl}
\pi^-_k(H)e_n&=& -2(k+n)e_n,\\ 
\pi^-_k(B)e_n&=& -\sqrt{n(2k+n-1)}e_{n-1} \quad\mbox{and}\quad\\
\pi^-_k(C)e_n&=& \sqrt{(n+1)(2k+n)}e_{n+1}, 
\end{array}
\end{equation}
see eq.\ (2.3) of Groenevelt \& Koelink (2002). The principal
series and complementary series representations are defined
on $\ell^2({\mathbb Z})$, a different Hilbert space. 
Koelink \& Van Der Jeugt (1998) (proposition 3.1, rephrased) 
have expressed the eigenvectors 
of $\pi^+_k(c\cdot H+B-C)$ for an arbitrary $c\in {\mathbb R}$
satisfying $|c|< 1$ in terms of the {\em orthonormal
Meixner-Pollaczek polynomials $p_n^{(\lambda)}(x;\phi)$}, given 
by 
$$p_n^{(\lambda)}(x;\phi)
=\sqrt{\frac{n!}{\Gamma(n+2\lambda)}}P_n^{(\lambda)}(x;\phi)$$
where the polynomials $P_n^{(\lambda)}(x;\phi)$ are the Meixner-Pollaczek
  polynomials.
\begin{theorem}[Koelink \& Van Der Jeugt]
\label{T_KVJ}
Introducing $X_\phi=-\cos \phi H+B-C$, the vector 
$$v(x):=\sum_{n\geq 0} p_n^{(k)} (x;\phi) e_n$$
is a generalized eigenvector for $\pi^+_k(X_\phi)$, with eigenvalue 
$2x\sin\phi$. 
\end{theorem}
An analogous result for $c=1$ was worked out by Van Der Jeugt (1997),
leading to Laguerre polynomials,  who also mentions (in his section
VIII) that the case $|c|>1$ leads to Meixner polynomials of the first
kind. The extension of Theorem~\ref{T_KVJ} to negative
discrete series representations is stated in proposition 3.1 of
Groenevelt \& Koelink (2002). An excellent source treating all cases
simultaneously is ch.\ 3 of Groenevelt's thesis (2004).

Klimyk (2006) used the positive series representations of $su(1,1)$ 
to study a model of a quantum oscillator. He defines $su(1,1)$
as the Lie algebra generated by 
$J_0$, $J_+$, $J_-$,  subject to the commutation relations 
$[J_0,J_+]=J_+$, $[J_0,J_-]=-J_-$ and $[J_-,J_+]=2J_0$.  
This definition is equivalent to the above one, via
setting $H:=2J_0$, $B:=J_+$ and $C:=-J_-$. Klimyk (2006) realizes the
orthonormal basis vectors  
$\{e_n\}_{n\geq 0}$ as polynomials 
$e_n^{k}(y):=\left(\frac{(2k+n-1)!}{n!}\right)^{1/2}y^n$ 
in the single variable $y$, and the Hilbert space $\ell^2({\mathbb
  Z}_{\geq 0})$ is realized as the closure of the space
of all polynomials in a single variable $y$, whereas the generators 
$J_0$, $J_+$, $J_-$ of $su(1,1)$ become differential operators. 
One of Klimyk's (2006) observations (eq.\ (17)) is that the
eigenfunctions of the momentum operator 
$J_1:=\frac{1}{2}(J_++J_-)$ are of the form
\begin{equation}
\label{E_Klimyk}
\psi_p(y)=\sum_{n\geq 0} P_n^{(k)}(p;\pi/2)y^n.
\end{equation} 
with the Meixner-Pollaczek polynomials $P_n^{(k)}(p;\pi/2)$. Obviously,
(\ref{E_Klimyk}) may be obtained from Theorem~\ref{T_KVJ} by
substituting $\phi=\pi/2$. By eq.\ (1.7.11) of   
Koekoek \& Swarttouw (1998), eq.\ (\ref{E_Klimyk}) is equivalent to  
$\psi_p(y)=(1-\imi y)^{-k+\imi\cdot
  p}(1+\imi y)^{-k-\imi\cdot p}$. 

Sukumar \& Hodges (2007) considered another model of a quantum
oscillator, using (besides others) the operators $L$, $R$, $S$, acting
on $\ell^2({\mathbb Z}_{\geq 0})$ as follows:
\begin{equation}
\label{E_SH}
\begin{array}{rcl}
Le_{2n}&=&\sqrt{(2n+1+\alpha)(2n+2))}/2 \cdot e_{2n+2}\\
Le_{2n-1}&=&\sqrt{(2n+1+\alpha)(2n))}/2 \cdot e_{2n+1}\\
Re_{2n+2}&=&\sqrt{(2n+1+\alpha)(2n+2))}/2 \cdot e_{2n}\\
Re_{2n+1}&=&\sqrt{(2n+1+\alpha)(2n))}/2 \cdot e_{2n-1}\\
Se_n&=&(2n+1+\alpha)/2 \cdot e_n\\
\end{array}
\end{equation}
The parameter $\alpha\in {\mathbb R}$ is assumed to satisfy $|\alpha|\leq 1$. 
As noted by Sukumar \& Hodges (2007) in eq. (2.1), the operators $L,R,S$
obey the same commutation rules as the generators $B,-C,H$ above, thus we
obtain a representation $\pi$ of $su(1,1)$ by setting 
\begin{equation}
\label{E_pidef}
\pi(B):=L, \quad \pi(C):=-R, \quad \pi(H):=S.
\end{equation} 
This representation may be
written as a direct sum of two representations by considering the
restrictions of the operators $L,R,S$ to the closure of the subspaces
generated by $\{e_{2n}\}_{n\geq 0}$ and $\{e_{2n+1}\}_{n\geq 0}$,
respectively. Introducing $\pi_{\beta}$ to denote the representation 
of $su(1,1)$ given by  
\begin{equation}
\label{E_SHbeta}
\begin{array}{rcl}
\pi_{\beta}(B)e_n&=&\sqrt{(n+1)(n+\beta)} e_{n+1},\\
\pi_{\beta}(C) e_n&=& -\sqrt{n(n-1+\beta)}  e_{n-1},\\
\pi_{\beta}(H) e_n&=& 2(n+\beta/2) e_n,\\
\end{array}
\end{equation}
it is not difficult to see that first restriction is equivalent to 
$\pi_{(\alpha+1)/2}$, whereas the second restriction is equivalent to 
$\pi_{(\alpha+3)/2}$. Comparing (\ref{E_SHbeta}) with (\ref{E_psr}) we
obtain 
$\pi_{\beta}=\pi^+_{\beta/2}$ for $\beta>0$.  
\begin{corollary}
\label{C_SHdec}
For $\alpha\in (-1,1]$,  the representation $\pi$ of $su(1,1)$ given by
\eqref{E_pidef} is isomorphic to the direct sum of the positive 
discrete discrete series representations $\pi^+_{(\alpha+1)/4}$ and
$\pi^+_{(\alpha+3)/4}$.
\end{corollary}
\begin{remark}\label{rem:nops}
For $\beta=0$, we get $\pi_{0}(B)e_0=\pi_{0}(C)e_0=\pi_{0}(H)e_0=0$,
thus we may  decompose $\pi_0$ as a direct sum of zero acting on the
linear span of $e_0$ and of the restriction of $\pi_0$ acting on the 
orthogonal complement $e_0^\perp$ of $e_0$. By shifting basis vector
indices in $e_0^\perp$ we may show that the restriction of
$\pi_0$ to $e_0^\perp$ is equivalent to $\pi_2=\pi^+_1$.
\end{remark}
It should be emphasized that corollary~\ref{C_SHdec} is applicable only to
the representation of $su(1,1)$, induced by the operators $L,R,S$. In
the work of Sukumar \& Hodges (2007) these operators are expressed in
terms of the annihilation operator $A$ and its adjoint
$A^{\dagger}$. The algebra generated by $A$ and $A^{\dagger}$ is a
representation of a Lie superalgebra properly containing $su(1,1)$. In this
paper we focus on the 
combinatorial statements related to the operators $L,R,S$ and
generalizations of these. In particular, Hodges \& Sukumar (2007) show 
(see eq.\ (5.2)) that for $\alpha=1$ we have
\begin{equation}
\label{E_ET}
\langle e_{0} , (L+R)^{2m} e_{0}\rangle =E_{2m}\quad\mbox{and}\quad 
\langle e_{1} , (L+R)^{2m}e_{1}\rangle =T_{2m+1}\quad\mbox{for all
  $m\geq 0$}. 
\end{equation}
Here the numbers $\{E_{2m}\}_{m\geq 0}$ resp.\
$\{T_{2m+1}\}_{m\geq 0}$ are the {\em secant} resp.\ {\em tangent}
numbers, given by 
$$
\sec(x)=\sum_{m\geq 0} \frac{E_{2m}}{(2m)!} x^{2m}
\quad\mbox{and}\quad
\tan(x)=\sum_{m\geq 0} \frac{T_{2m+1}}{(2m+1)!} x^{2m+1}.
$$
At the light of corollary~\ref{C_SHdec}, eq.\ \eqref{E_ET} is equivalent
to 
\begin{equation}
\label{E_ETe}
\langle e_0, \pi^+_{1/2}(B-C)^m e_0\rangle=E_{2m}
\quad\mbox{and}\quad 
\langle e_0, \pi^+_{1}(B-C)^m e_0\rangle=T_{2m+1}.
\end{equation}

\section{Matrix entries of tridiagonal operators on $\ell^2({\mathbb
    Z}_{\geq 0})$} 
\label{S_3d}

\begin{definition}
\label{D_3diag}
Consider the complex Hilbert space $\ell^2({\mathbb Z}_{\geq 0})$
with orthonormal basis vectors $e_{0}$, $e_{1}$, $e_{2}$, \ldots . 
We call a linear operator $T: D(T)\subset \ell^2({\mathbb Z}_{\geq
  0})\rightarrow \ell^2({\mathbb Z}_{\geq 0})$, defined on at least the
linear span of 
$\{e_{0}, e_{1}, e_{2}, \ldots,\}$, a {\em tridiagonal operator}, if
there exists real numbers $\{l_n\}_{n\geq 0}$, $\{d_n\}_{n\geq 0}$ and
$\{u_n\}_{n\geq 0}$ such that $Te_0=d_0e_0+l_0e_1$ and, for each 
$n\geq 1$, we have $Te_n=u_n e_{n-1}+d_n e_n+l_n e_{n+1}$.
\end{definition}
Equivalently, the restriction of $T$ to the linear
span of $\{e_{0}, e_{1}, e_{2}, \ldots,\}$ may be represented by a
tridiagonal matrix  
$$
\left(
\begin{array}{ccccccc}
d_0 & u_0 & 0 & 0 & 0 & 0& \ldots\\
l_0 & d_1 & u_1 & 0 & 0 & 0 & \ldots\\
  0 & l_1 & d_2 & u_2 & 0 & 0 & \ldots\\
  0 & 0  & l_2 & d_3& u_3 & 0 & \ldots\\
\vdots & & \ddots & \ddots & \ddots & \ddots\\
\end{array} 
\right)
$$
such that all entries are real numbers. 

The letters $\{l_n\}_{n\geq 0}$, $\{u_n\}_{n\geq 0}$ and $\{d_n\}_{n\geq
  0}$ should remind the reader of the words ``lower'', ``upper'', and
``diagonal'', as it is customary in numerical analysis. We define 
the {\em $LDU$-decomposition} of the tridiagonal operator $T$ as 
the sum of operators $T=L+D+U$, where the operators $L$, $D$, $U$
are given by $U e_0=0$ and
\begin{equation}
\label{E_LDU}
Le_n= l_n e_{n+1}, \quad  Ue_{n+1}= u_n e_n\quad
De_n= d_n e_n\quad \mbox{for $n\geq 0$}. 
\end{equation}
\begin{remark}
A it is a fortunate coincidence that the operator $L$ used by Sukumar \&
Hodges (2007) corresponds to the operator $L$ in this setting, modulo
corollary~\ref{C_SHdec}. The choice  
of $D$ will depend on what we want to compute: to recover
the formulas regarding the matrix entries of $L+R$ and $L+R+S$ in the
work of Sukumar \& Hodges (2007) and Hodges \& Sukumar (2007) we would
need to equate $D$ to zero or to the appropriate restriction of $S$.
\end{remark} 
The following lemma may be considered as a generalization of equations
(4.2) and (4.3) in Hodges \& Sukumar (2007). A reader making this
comparison should note that, for technical reasons, we read the words $X$
consisting of the letters $L,D,U$ {\em right-to-left}. 
\begin{lemma}
\label{L_Motzkin}
Let $X$ be a word of length $m$ consisting of the letters $L$, $D$, $U$,
where $T=L+D+U$ is the $LDU$-decomposition of a tridiagonal operator 
$T: D(T)\subset\ell^2({\mathbb Z}_{\geq 0})\rightarrow \ell^2({\mathbb
  Z}_{\geq 0})$  
and let $i\geq 0$ and $d\geq -i$ be integers. 
Associate to the pair $(X,i)$ a lattice path
starting at $(0,i)$ such that each $L$ is a
southeast step, each $U$ is a northeast step, each $D$ is an east step,
and we read the letters in $X$ right-to-left. Then $\langle e_{i+d}, X
e_i\rangle$ is not zero only if the lattice path associated to
$(X,i)$ is a Motzkin path of length $m$ from $(0,i)$ to
$(m,i+d)$. 
\end{lemma}
\begin{proof}
Applying $X$ to $e_i$ involves applying the
operators $L$, $D$, $U$ in it, reading $X$ from right to
left. While applying these operators, at every instance of the 
calculation our partial result is a scalar multiple of a single basis
vector. Each $L$
increases the index of this basis vector by $1$, each $U$ decreases it
by $1$ and each $D$ leaves the index unchanged. Thus the end result is 
either zero or a scalar multiple of $e_{i+d}$ where $d$ is
the difference between the number of $L$s and $U$s in $X$ (and may be
negative). If at any instance the number of $U$s read exceeds the
number of $L$s by more than $i$ then we get zero by
$Ue_{0}=0$. Therefore, if $\langle e_{i+d}, X e_i\rangle\neq 0$ 
then $(X,i)$ must represent a Motzkin path starting at $(0,i)$ and
necessarily ending at $(m,i+d)$.  
\end{proof}
The next lemma generalizes equations (4.6) and (4.7)
of Hodges \& Sukumar (2007). Our formulas look somewhat different mainly
because of our convention of reading $X$ right-to-left. 
\begin{lemma}
\label{L_Mrec}
For all $p,i\geq 0$ and $d\geq -i$ we have
\begin{align}
\label{E_M1}
\langle e_{i+d}, XDL^pe_i\rangle &= d_{p+i}\cdot \langle e_{i+d},
XL^pe_i \rangle
\quad\mbox{and}\\
\label{E_M2}
\langle e_{i+d}, XUL^{p+1}e_i\rangle
&= l_{p+i}u_{p+i} \cdot \langle e_{i+d}, XL^pe_i\rangle.
\end{align}
\end{lemma}
The proof is immediate.
Using eq.\ (\ref{E_hist}) and lemmas ~\ref{L_Motzkin} and~\ref{L_Mrec}
we may express each $\langle e_{i+d},T^m e_i\rangle$ in terms of a total
weight of weighted Motzkin paths. 
\begin{theorem}
\label{T_Motzkin}
Let $i\geq 0$ and $d\geq -i$ be integers, $T: D(T)\subset
\ell^2({\mathbb Z}_{\geq 0})\rightarrow \ell^2({\mathbb Z}_{\geq 0})$ a
tridiagonal operator and $T=L+D+U$ its $LDU$-decomposition. Then we have 
$$
\langle e_{i+d}, T^m e_i\rangle=
\begin{cases}
\sum_{\omega} \wt(\omega)&\mbox{if $d=0$,}\\
l_il_{i+1}\cdots l_{i+d-1}\sum_{\omega} \wt(\omega)&\mbox{if $d>0$,}\\
u_{i-1}u_{i-2}\cdots u_d\sum_{\omega} \wt(\omega)&\mbox{if $d<0$.}\\
\end{cases}
$$
Here $\sum_{\omega} \wt(\omega)$ is the total weight of all weighted
Motzkin paths of length $m$ starting at $(0, i)$, ending
at $(m,i+d)$ such that each northeast step has weight
$1$, each southeast step starting at level $n$ has weight
$\lambda_n=l_{n-1}u_{n-1}$, and each east step at level $n$ has weight
$d_{n}$.  In particular, if $D$ is identically zero then 
$\sum_{\omega} \wt(\omega)$ is the total weight of all weighted
Dyck paths of length $m$ starting at $(0,i)$, ending
at $(m,i+d)$ such that each northeast step has weight
$1$ and each southeast step starting at level $n$ has weight
$\lambda_n=l_{n-1}u_{n-1}$. 
\end{theorem}
As a consequence of theorem~\ref{T_Vmain}, setting $d=0$ and $i=0$ 
in theorem~\ref{T_Motzkin} yields the following algebraic statement. 
\begin{corollary}
\label{C_Motzkin}
If $l_{n}u_{n}\neq 0$ for all $n\geq 0$ then for all $m$,
$\langle e_{0}, T^m e_{0}\rangle$ is the $m$th moment of the functional
associated to the OPS given by 
(\ref{E_Favinit}) and (\ref{E_Favrec}) where $b_n=d_{n}$ and
$\lambda_n=l_{n-1}u_{n-1}$.
\end{corollary}
\begin{remark}
As noted in the Preliminaries, an OPS and the associated moments determine each
other only up to a constant factor. In corollary~\ref{C_Motzkin} we
understand that the moments were defined as total weights of 
weighted Motzkin paths as in theorem~\ref{T_Vmain}. If we obtain the
moments by some other means, we always need to check whether an
adjustment by a constant factor is necessary. It suffices to check
whether $\mu_0=1$ is satisfied.
\end{remark}
When we apply theorem~\ref{T_Motzkin} to $d\neq 0$,
theorem~\ref{T_Vinverse} becomes useful, at least for the case $i=0$. 
\begin{theorem}
\label{T_inverse}
For all $m,d\geq 0$ we have 
$$\langle e_{d}, T^m e_{0}\rangle=l_0l_1\cdots l_{d-1}\mu_{m,d},$$ 
where $\mu_{m,d}$ is the coefficient of $x^d$ in the degree $m$ inverse
polynomial for system of monic polynomials given by
(\ref{E_Favinit}) and (\ref{E_Favrec}) where $b_n=d_n$ and
$\lambda_n=l_{n-1}u_{n-1}$. 
\end{theorem}

\section{Matrix entries in discrete series representations of $su(1,1)$}
\label{S_su11adapt}

For a (positive, or negative) discrete series representation $\pi$ of
$su(1,1)$, the operator 
\begin{equation}
\label{E_Tc}
T_c[\pi,c]:=\pi(-C)+c\cdot \pi(H)+\pi(B)
\end{equation}
is a tridiagonal operator for any $c\in {\mathbb R}$.
Furthermore, the operators appearing in the $LDU$-decomposition of
$T[\pi,c]$ are 
$$
L[\pi]=\pi(-C),\quad D[\pi,c]=c\cdot \pi(H),\quad U[\pi]=\pi(B). 
$$
It is easy to see that $l_n u_n\neq 0$ holds in all cases, thus we may
apply corollary~\ref{C_Motzkin} to compute $\langle e_0 ,(\pi(-C)+c\cdot
\pi(H)+\pi(B))^m e_0\rangle $ for every discrete series representation 
$\pi$, all $c\in {\mathbb R}$ and all $m\in {\mathbb N}$. In particular, 
using (\ref{E_psr}) and (\ref{E_psn}) we obtain following corollary.   
\begin{corollary}
\label{C_Motzkinsu}
For any $k>0$, $m\in {\mathbb N}$ and $c\in {\mathbb R}$, the matrix
entry $\langle e_0 ,(\pi^+_k(-C)+c\cdot \pi^+_k(H)+\pi^+_k(B))^m
e_0\rangle$ is the $m$th moment of the functional associated to the OPS 
$\{p[k,c]_n (x)\}_{n\geq 0}$ given by (\ref{E_Favinit}) and
(\ref{E_Favrec}) where, 
\begin{equation}
\label{E_psrrec}
b_n=2c(k+n)\quad\mbox{and}\quad \lambda_n=n(n+2k-1).  
\end{equation}
Similarly, the matrix
entry $\langle e_0 ,(\pi^-_k(-C)+c\cdot \pi^-_k(H)+\pi^-_k(B))^m
e_0\rangle$ is the $m$th moment of the functional associated to the OPS 
$\{p[k,-c]_n (x)\}_{n\geq 0}$. 
\end{corollary}
The appropriate parts of the next proposition (restricted to the cases
$|c|=1$, $|c|>1$ and $|c|<1$, respectively) are either implied or 
explicitly stated in some equivalent form in the works of Van Der Jeugt
(1997), Koelink \& Van Der Jeugt (1998) and Groenevelt \& Koelink (2002).
A detailed and deep simultaneous discussion of all three cases
may be found in ch.\ 3 of Groenevelt's thesis (2004). Here we make a
summary statement with respect to our bases and outline the proof that
ensues from applying Meixner's classical method. 

\begin{proposition}[Groenevelt--Koelink--Van Der Jeugt]
\label{P_lmx}
The OPS $\{p[k,c]_n (x)\}_{n\geq 0}$ defined in corollary~\ref{C_Motzkinsu} 
is given by 
$$
p[k,c]_n (x)=
\begin{cases}
(\sqrt{c^2-1}-c)^n m_n\left(\frac{x}{2\sqrt{c^2-1}}-k; 2k,
  \frac{c-\sqrt{c^2-1}}{c+\sqrt{c^2-1}}\right)& \mbox{if
  $|c|>1$,}\\
L_n^{(2k-1)}(x) &\mbox{if $c=1$,}\\
(-1)^n L_n^{(2k-1)}(-x) &\mbox{if $c=-1$},\\
(\sqrt{1-c^2})^n\cdot M_n(x/\sqrt{1-c^2};c/\sqrt{1-c^2},2k) &\mbox{if
  $|c|< 1$.} 
\end{cases}
$$
Here the polynomials $\{L_n^{(2k-1)}(x)\}_{n\geq 0}$ are
Laguerre polynomials, the polynomials
$\{m_n(x;
2k,(c-\sqrt{c^2-1})/(c+\sqrt{c^2-1}))\}_{n\geq 0}$ and
$\{M_n(x;c/\sqrt{1-c^2},2k)\}_{n\geq 0}$, respectively,  
are Meixner polynomials of the first and second kind, respectively.
\end{proposition}
\begin{proof}
We apply Meixner's method to classify 
all OPS $\{P_n(x)\}_{n\geq 0}$ given by (\ref{E_Favinit}) and
(\ref{E_Favrec}) where 
\begin{equation}
\label{E_mxrec}
b_n=dn+f\quad\mbox{and}\quad \lambda_n=n(gn+h),   
\end{equation}
see section 3 of ch.\ 6 in Chihara (1978). By (\ref{E_psrrec})
we have
$$
d=2c, \quad f=2ck, \quad g=1,\quad h=2k-1. 
$$
The case $d^2-4g=4(c^2-1)>0$ is thus equivalent to $|c|>1$. Following 
section 3 of ch.\ 6 in Chihara (1978) we set
$\rho:=\sqrt{d^2-4g}=2\sqrt{c^2-1}$.  Note next that replacing $p[k,c]_n
(x)$ with 
$$q[k,c]_n(x):=p[k,c]_n(x-\gamma_{k,c})$$  
yields an OPS $\{q_n[k,c](x)\}_{n\geq 0}$ satisfying (\ref{E_mxrec}) with 
$$
d=2c, \quad f=2ck+\gamma_{k,c}, \quad g=1,\quad h=2k-1. 
$$
We want to choose $\gamma_{k,c}$ in such a way that 
$$
f=\frac{2(g+h)}{d+\rho},\quad\mbox{i.e.,}\quad
2ck+\gamma_{k,c}=\frac{4k}{2c+2\sqrt{c^2-1}} 
$$
is satisfied. This yields 
$$
\gamma_{k,c}=\frac{2k}{c+\sqrt{c^2-1}}-2kc=2k(c-\sqrt{c^2-1})-2kc=
-2k\sqrt{c^2-1} 
$$
and we
may apply the results stated in section 3 of ch.\ 6 in Chihara (1978) to 
$\{q_n[k,c](x)\}_{n\geq 0}$. The parameter ``$c$'' appearing in the
definition of the corresponding Meixner polynomials of the first kind becomes 
$$
\frac{d-\rho}{d+\rho}=\frac{2c-2\sqrt{c^2-1}}{2c+2\sqrt{c^2-1}}
=\frac{c-\sqrt{c^2-1}}{c+\sqrt{c^2-1}}. 
$$
and the parameter $\beta$ becomes $1+(2k-1)/1=2k$. We obtain
$$
\left(\frac{1}{\sqrt{c^2-1}-c}\right)^n q[k,c]_n (2\sqrt{c^2-1}x)
=
m_n(x;2k,(c-\sqrt{c^2-1})/(c+\sqrt{c^2-1})) 
$$
or, equivalently, 
$$
q[k,c]_n (x)
=(\sqrt{c^2-1}-c)^n m_n\left(\frac{x}{2\sqrt{c^2-1}}; 2k,
  \frac{c-\sqrt{c^2-1}}{c+\sqrt{c^2-1}}\right). 
$$
The formula for $p[k,c]_n(x)$ now follows from
$p[k,c]_n(x)=q[k,c]_n(x+\gamma_{k,c})$. 

The case $d^2-4g=4(c^2-1)<0$ is equivalent to $|c|<1$. The parameter
$\delta$ of the corresponding Meixner polynomials of the second kind
becomes $2c/\sqrt{4-4c^2}=c/\sqrt{1-c^2}$ and the parameter $\eta$
  becomes $1+(2k-1)/1=2k$.
We obtain
$$
\left(\frac{1}{\sqrt{1-c^2}}\right)^n p[k,c]_n(x\sqrt{1-c^2})
=
M_n(x;c/\sqrt{1-c^2},2k).
$$
The case $d^2-4g=4(c^2-1)=0$ is equivalent to $|c|=1$.
For $c=1$, (\ref{E_psrrec}) gives $b_n=2n+2k$ and $\lambda_n=n(n+2k-1)$, 
which define the monic Laguerre polynomials $p[k,1]_n(x)=L_n^{(2k-1)}(x)$. 
Similarly, we also have $p[k,-1]_n(x)=(-1)^n L_n^{(2k-1)}(-x)$. 
\end{proof}
\begin{remark} 
Inspired by theorem~\ref{T_KVJ} and
corollary~\ref{C_MPn}, for $|c|<1$ we may introduce a $\phi$ such that
$0<\phi<\pi$ and $c=\cos(\phi)$ hold. Then we obtain
\begin{equation}
p[k,\cos(\phi)]_n (x)=\sin(\phi)^{n}
M_n(x/\sin(\phi);\cot(\phi),2k).
\end{equation}
\end{remark}
As indicated in the preliminaries, a combinatorial interpretation of
the ``Laguerre case'' $|c|=1$ may be obtained using Viennot's (1983)
``Laguerre stories''. The ``best'' combinatorial interpretations (linked
to permutation enumeration with no special weights) is associated to the
case $k\in\{0,1\}$.  At the light of corollary~\ref{C_SHdec}, this
means that a combinatorial study of the matrix entries of the operator
$R+L+S$ appearing in the work of Sukumar \& Hodges (2007) is easiest
when  $(\alpha-1)/2$ or $(\alpha+1)/2$ belongs to 
$\{0,1\}$. As a consequence of (\ref{E_Lmoment}) we obtain
\begin{equation}
\langle e_{0}, (L+S+R)^ne_{0}\rangle =\left(\frac{\alpha+1}{2}\right)_n
\quad\mbox{for $\alpha\in (-1,1]$.}
\end{equation}
The case $\alpha=-1$ may be handled using
remark~\ref{rem:nops}. Similarly, for $e_1$ we have
\begin{equation}
\langle e_{1},(L+S+R)^ne_{1}\rangle =\left(\frac{\alpha+3}{2}\right)_n
\quad\mbox{for all $\alpha\in [-1,1]$.}
\end{equation}

Another combinatorially interesting case is $c=0$ where 
proposition~\ref{P_lmx} gives $p[k,0]_n(x)=M_n(x;0,2k)$. 
As a consequence, the operator $L+R$ appearing in the work of
Sukumar \& Hodges (2007) satisfies 
\begin{equation}
\label{E_L+Rmu0}
\langle e_{0}, (L+R)^{2n}e_{0}\rangle =\mu_{2n}(0,(\alpha+1)/2)
\quad\mbox{for $\alpha\in (-1,1]$.}
\end{equation}
The case $\alpha=-1$ may again be handled using
remark~\ref{rem:nops}. Similarly, for $e_1$ we have
\begin{equation}
\label{E_L+Rmu1}
\langle e_{1},(L+R)^{2n}e_{1}\rangle =\mu_{2n}(0,(\alpha+3)/2)
\quad\mbox{for all $\alpha\in [-1,1]$.}
\end{equation}
Here $\mu_{2n}(0,\eta)$ stands for the $(2n)$th moment associated to the
Meixner polynomials of the second kind with parameters $(0,\eta)$.
For $\eta=1$, (\ref{E_m2n}) takes a very simple form, as this is noted
in Example 21 of ch.\ II by Viennot (1983). Thus, by applying
corollary~\ref{C_Motzkinsu} to  $c=0$ and
$k=1/2$, we obtain a combinatorial explanation for the first half of
(\ref{E_ETe}). 
\begin{remark}
Viennot's (1983) work offers a nice model for the moments $\mu_{2n}(0,1)$
by further restricting the appropriate ``Laguerre stories''. 
This corresponds to enumerating certain zig-zag permutations.  For other
values of $\eta$ we could not avoid avoid considering ``weighted
stories'' in Viennot's (1983) model.
The case $\alpha=0$ in the work Sukumar \& Hodges (2007) does not
correspond to $\eta=1$, yet Sukumar \& Hodges (2007) find a simple,
non-weighted model by considering the enumeration of {\em transposition
  zig-zag classes}. It is an interesting question for future research
whether this enumeration question may be related to Viennot's (1983) theory by
  constructing a highly nontrivial bijection, similarly to Viennot's
  (1983) bijection between all permutations and his ``Laguerre stories''. 
\end{remark}
For other values of $\eta$ we may use lemma~\ref{L_Carlitz} which gives 
\begin{equation}
\label{E_ETg}
\langle e_0, \pi^+_{k}(B-C)^m e_0\rangle=E^{(2k)}_{2m}\quad\mbox{where}\quad
\sec ^{2k} (t)=\sum_{n\geq 0} E^{(2k)}_n \frac{t^n}{n!}
\end{equation}
In particular, substituting $k=1$ into (\ref{E_ETg}) yields the second
half of (\ref{E_ETe}). Indeed, as noted in eq.\ (3.2) of Sukumar \&
Hodges (2007), we have $E^{(2)}_n=T_{2n+1}$. Using (\ref{E_ETg}) for an
arbitrary parameter $k$ allows us to rewrite (\ref{E_L+Rmu0}) and
  (\ref{E_L+Rmu1}) as 
\begin{equation}
\langle e_{0},(L+R)^{2n}e_{0}\rangle=E^{((\alpha+1)/2)}_{2n}
\quad\mbox{for $\alpha\in (-1,1]$ and }
\end{equation}
\begin{equation}
\langle e_{1},(L+R)^{2n}e_{1}\rangle =E^{((\alpha+3)/2)}_{2n}
\quad\mbox{for all $\alpha\in [-1,1]$}.
\end{equation}

Finally, in the case $|c|>1$, leading to Meixner polynomials of the first
kind, we may use (\ref{E_mx1mom}) to find $\langle e_0
,(\pi^+_k(-C)+c\cdot \pi^+_k(H)+\pi^+_k(B))^m e_0\rangle$. We omit the
details for this case and mention only one, combinatorially interesting,
subcase. For $k=1/2$ and $c=3/\sqrt{8}$ the parameter $\beta$ of the
associated Meixner polynomial becomes $1$, whereas the parameter ``$c$''
becomes $1/2$. As it is explained in ch. II, \S 8 of Viennot's work
(1983), the unitary Meixner polynomials of the first kind
$\{\hat{m}_n(x;1,1/2)\}_{n\geq 0}$ are the {\em Kreweras polynomials}
whose moments $\mu_n$ are the numbers of {\em bicolored permutations} of
$\{1,2,\ldots,n\}$ where the two colors are used to color the descents
of each permutation.  

What we have seen thus far is that the application of 
corollary~\ref{C_Motzkin} to positive discrete series representations of
$su(1,1)$ is linked to the use of Laguerre
polynomials and of Meixner polynomials of both kinds. The same orthogonal
polynomial sequences need to be used to apply theorem~\ref{T_Vinverse}
to the same representations. The only difference is that, in each case,
instead of the moments we would need to use the inverse polynomials of
the same OPS. We may find these inverse
polynomials using (\ref{E_Sdef}) and (\ref{E_Sinv}). For brevity's sake,
in the rest of the section we only outline how to find the inverse
polynomials of Laguerre and Meixner polynomials, and leave the
adaptation of these data to theorem~\ref{T_Vinverse} to the reader. 

The Laguerre polynomials $\{L_n^{(\alpha)}(x)\}_{n\geq 0}$
satisfy (29) in ch.\ II of Viennot's (1983) work:
\begin{equation}
\sum_{n\geq 0} L_n^{(\alpha)}(x) \frac{t^n}{n!}
=(1+t)^{-\alpha-1}\exp\left(x \frac{t}{1+t}\right),
\end{equation}
thus we must set $f(t)=(1+t)^{-\alpha-1}$ and $g(t)=t/(1+t)$ in
(\ref{E_Sdef}). As indicated in table 4 of ch.\ III by Viennot (1983), this
implies $g^{\langle -1\rangle }(t)=t/(1-t)$. Thus the inverse polynomials of
$\{L_n^{(\alpha)}(x)\}_{n\geq 0}$ are given by 
\begin{equation}
\label{E_Linv}
\sum_{n\geq 0} Q_n(x) \frac{t^n}{n!}
=\left(\frac{1}{1-t}\right)^{\alpha+1}\exp\left(x \frac{t}{1-t}\right).
\end{equation}
Hence we have
\begin{align*}
\mu_{n,d}&=[x^d] Q_n(x)
=\frac{n!}{d!}\sum_{j=0}^{n-d} (-1)^j\binom{-\alpha-1}{j}
\binom{-d}{n-d-j}(-1)^{n-d-j}\\ 
&=\frac{(-1)^{n-d}n!}{d!}\binom{-d-\alpha-1}{n-d},
\end{align*}
implying
\begin{equation}
\label{E_Linvc}
\mu_{n,d}=\binom{n}{d}(\alpha+1+d)_{n-d}.
\end{equation}

The generating function for the  unitary Meixner polynomials 
$\{\hat{m}_n(x;\beta,c)\}_{n\geq 0}$ may be found in table 4 of ch.\ III
by Viennot (1983): 
$$
\sum_{n\geq 0} \hat{m}_n(x;\beta,c)
\frac{t^n}{n!}=\left(1+\frac{tc}{1-c}\right)^{-\beta}
\left(\frac{1-c+t}{1-c+ct}\right)^{x},$$
thus we must set $f(t)=(1+tc/(1-c))^{-\beta}$ and $g(t)=\ln
(1-t+c)-\ln(1-c+ct)$ in (\ref{E_Sdef}). 
As indicated in table 4 of ch.\ III by Viennot (1983), this
implies $g^{\langle -1\rangle }(t)=(1-c)(\exp(t)-1)/(1-c\exp(t))$, and
the inverse polynomials of $\{\hat{m}_n(x;\beta,c)\}_{n\geq 0}$
are given by 
\begin{equation}
\label{E_M1inv}
\sum_{n\geq 0} Q_n(x) \frac{t^n}{n!}
=\left(\frac{1-c}{1-c\exp(t)}\right)^{\beta}
\exp\left(\frac{x(1-c)(\exp(t)-1)}{1-c\exp(t)}\right).
\end{equation}

The Meixner polynomials of the second kind $\{M_n(x;,\delta,\eta)\}_{n\geq 0}$
satisfy (69)  in ch.\ II of Viennot's (1983) work:
$$
\sum_{n\geq 0} M_n(x;\delta,\eta)\frac{t^n}{n!}
=((1+\delta t)^2+t^2)^{-\eta/2}\exp\left(x\arctan
\left(\frac{t}{1+\delta t}\right)\right),
$$
thus we must set $f(t)=((1+\delta t)^2+t^2)^{-\eta/2}$ and
$g(t)=\arctan(t/(1+\delta t))$ in
(\ref{E_Sdef}). As indicated in table 4 of ch.\ III by Viennot (1983), this
implies $g^{\langle -1\rangle }(t)=\tan(t)/(1-\delta \tan t)$, and 
the inverse polynomials of $\{M_n(x;,\delta,\eta)\}_{n\geq 0}$ are given by 
\begin{equation}
\label{E_Minv}
\sum_{n\geq 0} Q_n(x) \frac{t^n}{n!}
=\left(\frac{\sec(t)^2}{(1-\delta\tan(t))^2}\right)^{\eta/2}\exp\left(x
\frac{\tan(t)}{1-\delta \tan(t)}\right).
\end{equation}
In particular, we have
\begin{equation}
\label{E_Minvc}
\mu_{n,d}=\frac{n!}{d!} [t^n] \sec(t)^{\eta}\tan(t)^d \quad\mbox{for
  $\delta=0$}.
\end{equation}

\section{Eigenvectors of closable tridiagonal operators}
\label{S_spectrum}

In this section we return to general tridiagonal operators on the
complex Hilbert space $\ell^2({\mathbb Z}_{\geq 0})$, but it will be
convenient to assume that they are {\em closable}.
This assumption was not necessary in section~\ref{S_3d}, where we were
only interested in calculating the matrix entries in the first column 
for the powers of tridiagonal operators. Formally, we are still able to state
theorem~\ref{T_spectrum} below without assuming closability. However, 
the theorem implies that the eigenvectors of tridiagonal operators of
the stated form must arise as infinite linear combinations of the basis 
vectors. In all applications, it seems easier to show that such vectors 
belong to the domain of the operator when the operator is closable.

As noted at the end of section 2.1 by Groenevelt \& Koelink (2002), 
for any (positive or negative) discrete series
representation $\pi$ of $su(1,1)$, the operators $\pi(B)$, $\pi(C)$ and
$\pi(H)$ are closable. As a consequence, the operators  
operators $T_c$ defined in (\ref{E_Tc}) are closable. They are even
self-adjoint, i.e., {\em Jacobi operators}. We refer the reader to
Koelink (2004) for the spectral theory of Jacobi operators. This theory was
used by Van Der Jeugt (1997), Koelink \& Van Der Jeugt (1998) and
Groenevelt \& Koelink (2002) to obtain a complete description of the spectra 
of the operators $T_c$. The main theorem of this section 
generalizes only the necessity part of their results. Even this
generalization suffices to highlight the fact: the (potential)
eigenvectors of any (closable) tridiagonal operator $T$ may be
expressed using the same OPS, whose moments and and inverse polynomial
coefficients may be used to express the first columns in the matrices of
powers of $T$.   
    
\begin{theorem}
\label{T_spectrum}
Let $T$ be a tridiagonal operator with $LDU$-decomposition $T=L+D+U$
such that $l_n u_n\neq 0$ holds for all $n\geq 0$ and let $z\in{\mathbb C}$
any complex number. If $z$ is an eigenvalue of $T$ then the associated
eigenspace is one dimensional, generated by 
$$
v_z:=\sum_{n\geq 0} \frac{p_{n}(z)}{\prod_{i=0}^{n-1} u_{i}} e_{n}.
$$
Here $\{p_n(x)\}_{n\geq 0}$ is the  sequence of monic polynomials defined by
(\ref{E_Favinit}) and (\ref{E_Favrec}), satisfying 
$b_n=d_n$ and $\lambda_n=u_{n-1}l_{n-1}$.
\end{theorem} 
\begin{proof}
A vector of the form $\sum_{n\geq 0} h_n\cdot e_{n}$
is eigenvector of $T$ associated to the eigenvalue $z$ only if 
\begin{equation}
\label{E_Einit}
d_0h_0+u_0h_1=zh_0\quad\mbox{and}
\end{equation}
\begin{equation}
\label{E_Erec}
l_{n-1} h_{n-1}+d_nh_n+u_{n} h_{n+1}=z h_n \quad\mbox{holds for $n\geq 1$}.
\end{equation}
Introducing $p_0:=h_0$ and 
$$
p_n:=h_{n}\cdot u_0u_1\cdots u_{n-1}\quad\mbox{for $n\geq 1$} 
$$
we may rewrite (\ref{E_Einit}) as 
\begin{equation}
\label{E_Epinit}
p_1=(z-d_0) p_0,
\end{equation}
while multiplying both sides of (\ref{E_Erec})
by $u_0u_1\cdots u_{n-1}$ yields the equivalent equation
\begin{equation}
\label{E_Eprec}
p_{n+1}=(z-d_n) p_{n}-u_{n-1}l_{n-1} p_{n-1}\quad\mbox{for $n\geq 1$}.
\end{equation}
If $p_0=0$ then (\ref{E_Epinit}) and (\ref{E_Eprec}) implies $p_n=0$ for
all $n$. We obtain a nonzero eigenvector only if $p_0\neq 0$ and then,
without loss of generality we may assume $p_0=1$. Furthermore,
(\ref{E_Eprec}) implies that $p_n$ must be obtained by substituting $z$
into a monic polynomial sequence $\{p_n(x)\}_{n\geq 0}$ satisfying
Favard's recursion formula with $b_n=d_n$ and $\lambda_n=u_{n-1}l_{n-1}$
for $n\geq 1$. Therefore any eigenvector associated to $z$ is of the
form stated in the theorem.
\end{proof} 
\begin{remark}
In the special case when the tridiagonal
operator $T$ is a Jacobi operator, $T$ is self-adjoint and thus has
only real eigenvalues. 
\end{remark}
The connecting coefficients $\{\lambda_n\}_{n\geq 1}$ resp.\
  $\{b_n\}_{n\geq 0}$ for the polynomials $\{p_n(x)\}_{n\geq 0}$ 
used in theorem~\ref{T_spectrum} above are the same as
  the ones used in corollary~\ref{C_Motzkin}
  and theorem~\ref{T_inverse}. Thus, the same OPS may be
  used to calculate the first column of matrix entries of the powers of $T$ 
as in the calculation of its eigenvectors. In particular, for any 
  operator $T_c$ defined in (\ref{E_Tc}) we must use the appropriate 
 OPS $\{p[k,c]_n(x)\}_{n\geq 0}$ from proposition~\ref{P_lmx}. This
  leads to the use of Laguerre polynomials, and Meixner polynomials of
  both kinds, as in the above cited results of  
Van Der Jeugt (1997), Koelink \& Van Der Jeugt (1998) and
Groenevelt \& Koelink (2002). 

We conclude this section with an observation regarding the case $c=0$
and $k=1/2$, leading to $p[1/2,0](x)=M_n(x;0,1)$, by proposition~\ref{P_lmx}.
A combinatorial model for the coefficients of the powers of $x$ in
$M_n(x;0,1)$ is mentioned in Example 21 of ch.\ II by Viennot (1983). Up to
sign, the coefficient of $x^j$ in $M_n(x;0,1)$ is the number of
permutations of $\{1,2,\ldots,n\}$ having $j$ odd cycles. Thus, for this
specific choice of parameters, not only the selected matrix entries of
the powers of $T_0$, but also the description of the eigenvectors 
has a combinatorial interpretation. Recall that the same selection of
parameters leads to the first part of (\ref{E_ETe}), which is equivalent 
to the first part of (\ref{E_ET}) that also caught the attention of
Hodges \& Sukumar (2007). Therefore, if someone wanted to build a
``purely combinatorial model'' of a quantum oscillator in the future,
starting with $\alpha=1$ in the model proposed Hodges \& Sukumar (2007)
or with $k=1/2$ in the model proposed by Klimyk (2006), and focusing on
understanding the action of the moment operator seems a very reasonable
first step.  

\section{Beyond self-adjoint operators and beyond $su(1,1)$}
\label{S_ex}

Up to this point, our applications of theorems~\ref{T_Motzkin},
\ref{T_inverse} and \ref{T_spectrum} involved self-adjoint operators
from discrete series representations of $su(1,1)$.
In this final section we outline two examples indicating that 
that these theorems may also be useful in other applications where the 
operator is not self-adjoint or where it belongs to the representation
of an algebra that is is not $su(1,1)$. 

Our first example is $\pi^{+}_{k}(B+C)$ where $\pi^{+}_k$ is a positive
discrete series representation of $su(1,1)$ for some $k>0$. 
In analogy 
to corollary~\ref{C_Motzkinsu}, the matrix
entry $\langle e_0 , \pi^+_k(B+C)^m e_0\rangle$ is the
$m$th moment of the functional associated to the OPS 
$\{f_n^{(2k)} (x)\}_{n\geq 0}$ given by (\ref{E_Favinit}) and
(\ref{E_Favrec}) where 
$$
b_n=0\quad\mbox{and}\quad \lambda_n=-n(n+2k-1).  
$$
It is easy to verify that we have 
\begin{equation}
\label{E_iL}
f_n^{(2k)}(x)=(-\imi)^{n} M_n(\imi x;0,2k)
\end{equation}
where the polynomials $M_n(x;0,2k)$ are Meixner polynomials 
of the second kind. Thus, for $k\in \frac{1}{2} {\mathbb Z}_{\geq 0}$,
the polynomials $f_n^{(2k)}(x)$  
are identical to the polynomials studied by Carlitz (1959). 
The applications of theorems~\ref{T_Motzkin} and~\ref{T_inverse} are no
different from the self-adjoint setting. By theorem~\ref{T_spectrum},
a complex number $z\in {\mathbb C}$ is an
eigenvalue $\pi^{+}_{k}(B+C)$ only if the expression
$$
v^{(k)}_z:=\sum_{n\geq 0} \frac{f_n^{(2k)}(z)}{\sqrt{(n-1)! (2k)_n}} e_n
$$
represents a vector belonging to $\ell^2(\mathbb{Z}_{\geq 0})$. The
the discussion of the convergence issues is analogous to the case 
of $\pi^+_k(B-C)$ that was worked out by Koelink \& Van Der Jeugt
(1998): here we obtain that all eigenvalues belong to $\imi{\mathbb R}$. 
This is not surprising, since $\imi\pi^{+}_k(B+C)$ is self-adjoint.
However, the way were able to handle the example indicates that the
validity of our main results is not restricted to self-adjoint operators. 

Our second example is the sum $A+A^{\dagger}$ where $A$ and
$A^{\dagger}$ are the operators defined in eq.\ (2.1) of Sukumar \&
Hodges (2007):
$$
\begin{array}{lclclcl}
A e_{2n}&=&\sqrt{2n} \cdot e_{2n-1},&\quad & A e_{2n+1}&=&\sqrt{2n+1+\alpha}
\cdot e_{2n},\\ 
A^{\dagger} e_{2n}&=&\sqrt{2n+1+\alpha} \cdot e_{2n+1}, &\quad &
A^{\dagger} e_{2n-1}&=&\sqrt{2n} \cdot e_{2n}.\\
\end{array}
$$
The operator $A+A^{\dagger}$ is self-adjoint, but it belongs to a
representation of a Lie superalgebra properly containing $su(1,1)$. 
Theorems~\ref{T_Motzkin}, \ref{T_inverse} and
\ref{T_spectrum} are obviously applicable, and the OPS
$\{p_n(x)\}_{n\geq 0}$ appearing in all of them is given by 
(\ref{E_Favinit}) and (\ref{E_Favrec}) where, 
$$
b_n=0, \quad \lambda_{2n+1}=2n+1+\alpha,\quad\mbox{and}\quad
\lambda_{2n}=2n. 
$$
We obtain the {\em generalized Hermite polynomials}, see
eq.\ (2.46) of ch.\ V in Chihara (1978). Choosing $\alpha=0$ yields the
usual Hermite polynomials. Sukumar \& Hodges (2007) note that this choice
corresponds to the standard harmonic oscillator. For nonzero 
$\alpha$, the parity-dependent nature of the quantum algebra defined by
Sukumar \& Hodges (2007) is reflected in the parity-dependent 
nature of the recursion defining the generalized Hermite polynomials.

\section*{Acknowledgements} 
I wish to thank Mireille Bousquet-M\'elou for acquainting me with
Viennot's (1983) work and two anonymous referees whose suggestions
greatly improved the content and presentation of this paper. Many thanks
to Erik Koelink and Alan Lambert who promptly answered many of my questions. 
This work was supported by the NSA grant
\# H98230-07-1-0073 and essentially completed while the author was on
reassignment of duties sponsored by the University of North Carolina at
Charlotte.

\end{document}